\documentclass[11pt]{amsart}
\usepackage{color} 
\usepackage{url}
\usepackage[cmtip, all]{xy}

\usepackage[export]{adjustbox}
\usepackage{graphicx}

\usepackage[normalem]{ulem}

\theoremstyle{plain}
\newtheorem{thm}{Theorem}[section]

\newtheorem{lemma}[thm]{Lemma}

\newtheorem{proposition}[thm]{Proposition}
\theoremstyle{definition}
\newtheorem{remark}[thm]{Remark}

\numberwithin{equation}{section}

\newcommand{\ga}[2]{\begin{gather}\label{#1}#2 \end{gather}}

\newcommand{\sA}{{\mathcal A}}

\newcommand{\sC}{{\mathcal C}}

\newcommand{\sE}{{\mathcal E}}

\newcommand{\sK}{{\mathcal K}}
\newcommand{\sL}{{\mathcal L}}

\newcommand{\sO}{{\mathcal O}}

\newcommand{\sS}{{\mathcal S}}

\newcommand{\sV}{{\mathcal V}}

\newcommand{\sX}{{\mathcal X}}

\newcommand{\C}{{\mathbb C}}

\newcommand{\F}{{\mathbb F}}
\newcommand{\G}{{\mathbb G}}
\renewcommand{\H}{{\mathbb H}}

\newcommand{\N}{{\mathbb N}}

\newcommand{\Q}{{\mathbb Q}}

\newcommand{\Z}{{\mathbb Z}}

\title [Cohomologically rigid local systems and integrality]{Cohomologically rigid local systems and integrality}
\author{H\'el\`ene Esnault \and Michael Groechenig} 
\address{Freie Universit\"at Berlin, Arnimallee 3, 14195, Berlin,  Germany}
\email{esnault@math.fu-berlin.de}
\address{Freie Universit\"at Berlin, Arnimallee 3, 14195, Berlin,  Germany}
\email{groemich@zedat.fu-berlin.de}
\thanks{The first  author is supported by  the Einstein program and the ERC
Advanced
Grant 226257, the second author was supported by a Marie Sk\l odowska-Curie fellowship: This project has received funding from the European Union's Horizon 2020 research and innovation programme under the Marie Sk\l odowska-Curie Grant Agreement No. 701679. \\ \includegraphics[height = 1cm,right]{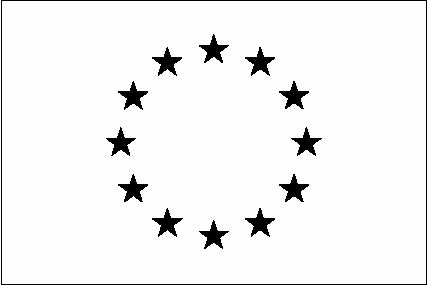}
}

\date{ January 2nd, 2018}
\begin{document}
\begin{abstract}
We prove that the monodromy of an  irreducible   cohomologically complex rigid  local system with finite determinant  and quasi-unipotent  local monodromies at infinity on  a smooth quasiprojective complex  variety $X$  
is integral. This answers positively  a special case of a conjecture by Carlos Simpson.  On a smooth projective variety, the argument relies on Drinfeld's theorem on the existence of $\ell$-adic companions over a finite field.  When the variety is quasiprojective, one has in addition to control the weights and the monodromy  at infinity.
\end{abstract}
\maketitle

\section{Introduction}\label{intro}
Let $X $ be a smooth connected quasiprojective complex  variety, $j: X\hookrightarrow \bar X$ be a {\it good compactification}, that is a  smooth compactification such that $D=\bar X\setminus X$ is a  normal crossings divisor.  
An irreducible complex local system 
 $\sV$   is said to be {\it cohomologically rigid} if  $\H^1(\bar X, j_{!*}\sE nd^0(\sV))=0.$
The finite dimensional complex vector space $\H^1(\bar X, j_{!*}\sE nd^0(\sV))$
 is the Zariski tangent space at the moduli point of $\sV$  of the Betti moduli stack of complex local systems of rank $r$ with prescribed determinant and  prescribed local monodromies along the components of $D$ (see Section~\ref{sec:mod}).  So a cohomologically rigid  complex local system is {\it rigid}, that is its moduli point is isolated, and in addition it is smooth. 
 
Simpson conjectures that {\it rigid irreducible  complex local systems with torsion determinant and quasi-unipotent monodromies  around the components of $D$  are of geometric origin}. In particular, they should be {\it integral}, 
where a complex local system is said to be integral if it is coming by extension of scalars from a local system of projective  $\sO_L$-modules of finite type, where  $\sO_L$ is the ring of integers of 
 a number field $L\subset \C$.
 See  \cite[Conj.~0.1, ~0.2]{LS16} for the formulation of the conjectures  in the projective case. 
 We prove:
 
\begin{thm} \label{thm:int}
 Let $X$ be a smooth connected   quasiprojective  complex variety.  Then irreducible cohomologically rigid complex  local systems with finite determinant and  quasi-unipotent  local monodromies around the components at infinity of a good compactification   are integral. 

\end{thm}

When $X$ is projective, a first proof of Theorem~\ref{thm:int} using $F$-isocrystals and $p$ to $\ell$-companions  (\cite{AE16}) has been given in \cite{EG17}. 
The short proof presented in this note  only uses  the $\ell$ to $\ell'$ companions, the existence of which has been proved by Drinfeld (\cite[Thm.~1.1]{Dri12}). 
We outline it in the projective case.

We use the following criterion for integrality. Let $\sV$ be  rigid complex local system   which comes
 by extension of scalars from a local system $\sV_{\sO_{K,\Sigma}}$ of projective $\sO_{K,\Sigma}$-modules where 
 $K\subset \C$ is a number field, $\Sigma$ is a finite set of places of $K$ and $\sO_{K,\Sigma}$ is the ring of $\Sigma$-integers of $K$.  
  For any place $\lambda$ of $K$, let $K_\lambda$ be the completion of $K$ at $\lambda$,  $K_\lambda \subset \bar K_{\lambda}$ be an algebraic closure, $\sO_{K_\lambda}\subset \bar \sO_{K_\lambda}$ be the underlying extension of rings of integers. Then $\sV$ is integral if and only if for any $\lambda$ in $\Sigma$, $\sV_{\bar K_\lambda}=\sV_K\otimes_K \bar K_\lambda$ comes by extension of scalars from a local system 
$\sV_{\bar \sO_{K_\lambda}}$ of free $\bar \sO_{K_\lambda}$-modules.  The local system of projective $\sO_K$-modules is  defined as the inverse image of  $\prod_{\lambda \in \Sigma} \sV_{\bar \sO_{K_\lambda}} \subset 
\prod_{\lambda \in \Sigma} \sV_{\bar K_\lambda}$ via the localization map 
 $\sV_{\sO_{K, \Sigma}}\to \prod_{\lambda \in \Sigma} \sV_{\bar K_\lambda} .$
(See \cite[Cor.~2.3, Cor.~2.5]{Bas80} for the same criterion expressed in term of  traces of representations).

  We fix natural numbers $r$  and $d$.  The moduli stack of  irreducible complex local systems of rank $r$ and determinant of order $d$ is of finite type. Thus
  the set of 
   isomorphism classes of  such local systems which are cohomologically rigid has finite cardinality  $N(r,d)$. 
   For any algebraic closed field $C$ of characteristic $0$,  $N(r,d)$ is also the number of  isomorphism classes of  irreducible $C$-local systems of rank $r$ and determinant of order $d$.

  The $N(r,d)$ complex local systems come by extension of scalars of local systems $\sV_K$ defined  over a number field $K\subset \C$.  As the topological fundamental group is finitely generated, the $N(r,d)$ local systems  $\sV_K$  come by extension of scalars from  local systems 
$\sV_{\sO_{K,\Sigma}}$ of free $\sO_{K,\Sigma}$-modules, where $\Sigma$ is a finite set of places of $K$ and $\sO_{K,\Sigma}$ is the ring of $\Sigma$-integers of $K$.  We want to show that $\sV_{\bar K_\lambda} = \sV_{ K_\lambda} \otimes_{ K_\lambda  } \bar K_{\lambda}$  comes from a local system defined over $\sO_{\bar K_\lambda}$ for all $\lambda$ in $\Sigma$.

   One chooses a place $\lambda$ of $K$ which is not  in $\Sigma$.  It divides a prime number $\ell$.
We complete the finitely many  $
\sV_{ \sO_{K,\Sigma} }$ considered at $\lambda$ so as to obtain  $\lambda$-adic lisse sheaves on $X$.
  We take a model $X_S$ of $X$ over  a scheme  $S$ of finite type over $\Z$ such that $X_S/S$ is smooth, and take a closed point $s$ of $S$ of characteristic $p$
   prime to the order of the residual representations, to $\ell$,
   to $d$, and to the residual characteristics of the places in $\Sigma$. 
   Then 
 the $\lambda$-adic sheaves descend to  lisse sheaves on $X_{\bar s}$, where $ k(s)\hookrightarrow k(\bar s)\simeq \bar \F_p$ is  an algebraic closure of the residue field of $s$.

  By a variant of an argument of Simpson (see Proposition~\ref{prop:simpson}), after finite base change on $s$, the lisse sheaves descend to  arithmetic lisse sheaves. They are   irreducible on $X_{\bar s}$ and can be taken to have finite determinant. 
  By Drinfeld's existence theorem, for all places $\lambda'$ of $K$ not dividing $p$, in particular for all $\lambda'$ in $\Sigma$, they have $\lambda'$-companions, which are $\bar K_{\lambda'}$-lisse sheaves, thus by definition, 
   come from $\bar \sO_{K_{\lambda'}}$-lisse sheaves.
 Using purity, the $L$-function of the trace $0$ endomorphisms of the companions, base change, local acyclicity, and Betti to \'etale comparison, one shows that viewed back as  representations of the topological fundamental group with values in $GL(r, \bar \sO_{K_{\lambda'}})$, the $\lambda'$-companions
 define on $X$ the required number $N(r,d)$ of cohomologically rigid non-isomorphic $\bar K_{\lambda'}$-local systems of the type we want, all coming from $\bar \sO_{K_{\lambda'}}$-local systems.

To generalize the argument to the quasiprojective case and quasi-unipotent monodromies around the components at infinity, 
one needs that specializiation,  purity,  
 base change, local acyclicity are applicable in the non-proper case. 
 This being acquired, 
 we bound the problem by  bounding the rank $r$, and fixing two natural numbers $d$ and $h$ such that the order of the determinant divides $d$ and the order of the eigenvalues of the monodromies around the components at infinity divides $h$.  There are finitely many isomorphism classes of cohomologically rigid local systems with those invariants, say $N(r,d,h)$. 
 Deligne's theorem on compatible systems on curves \cite[Thm.~9.8]{Del73b}  implies that those data are preserved by passing to companions, which enables us to make the counting argument with $N(r,d)$ replaced  by $N(r,d,h)$.

\medskip
{\it Acknowledgements:}  Theorem~\ref{thm:int}
was initially written only in the projective case. We thank Pierre Deligne for suggesting to us the generalization   presented here, allowing quasi-unipotent monodromies  along the components at infinity.   Even if most of  the arguments presented in this new version are variants of the ones contained in our initial version,  Proposition~\ref{prop:coh} and Remark~\ref{rmk:IC} are due to him,  and  are crucial for applying the ideas developed in the projective case to the case of quasi-unipotent monodromies at infinity. 
Furthermore, for clarification, 
 he sent us his version
 of Simpson's crucial theorem \cite[Thm.~4]{Sim92},  together with his version of Proposition~\ref{prop:coh}, which we need to later develop our weight argument. Beyond the mathematics, we thank him for his commitment 
 to our article.  
  We thank Carlos Simpson for exchanges on his general conjecture and Luc Illusie for kindly giving us the reference  for local acyclicity.

\section{Cohomologically rigid local systems with prescribed monodromy around the components of the divisor at infinity} \label{sec:mod}
Let $X$ be a smooth connected  quasiprojective  complex variety, $j: X\hookrightarrow \bar X$ be a good compactification, $D=\bar X\setminus X$ be the  normal crossings divisor at infinity. We write $D=\cup_{i=1}^N D_i$ where the $D_i$ are the irreducible components. 
In this section we clarify the notion of {\it rigid irreducible  complex local systems with torsion determinant and quasi-unipotent monodromies  around the components of $D$.}

Let $U=\bar X\setminus D_{\rm sing}$, where $D_{\rm sing}$ is the singular locus of $D$, and $j: X\xrightarrow{a}U \xrightarrow{b} \bar X$ be the open embeddings.   We choose 
a complex point $x\in X$ and  $r \in \N_{\ge 1}$.  For each $i=1,\ldots, N$, we fix a conjugacy class $\sK_i \subset GL(r,  \C)$.  It is the  set of complex points of a subvariety of $GL(r)$  which we denote by the same letter $\sK_i$. 
 If $\Delta_i\subset U$ is a ball around $y_i\in D_i \cap U$ and    $x_i \in \Delta^\times_i$ with 
$ \Delta_i^\times =\Delta_i\setminus D_i\cap \Delta_i$, the fundamental group 
$\pi_1(\Delta_i^\times, x_i) $ is freely generated  by the local monodromy  $T_i$, determined uniquely up to sign,  around $D_i \cap \Delta_i$, so
$\pi_1(\Delta_i^\times, x_i) =\Z\cdot T_i$. 
Any  complex linear local system $\sV$ defines by restriction a complex  linear local system $\sV|_{\Delta^\times_i}$ on $\Delta_i^\times$.  We say  that  $\sV|_{\Delta_i^\times}$ {\it  is defined by} $\sK_i$  if  the image of $T_i$ lies in $\sK_i$  in its monodromy representation.   We say   that  $\sV$ {\it  is defined by} $\sK_i$  {\it along} $ D_i$  if $\sV|_{\Delta^\times_i}$ is 
{\it for all points} $y_i\in D_i\cap U$.  
As $D_i$ is irreducible, $D_i\cap U$ is smooth and connected, so the condition that $\sV$ be  defined by $\sK_i$ 
along $D_i$ is equivalent to the condition that $\sV|_{\Delta_i^\times}$ is defined by $\sK_i$ for the choice of one single $y_i  \in D_i\cap U$.   For each $i=1,\ldots, N$, we fix one such $y_i\in D_i\cap U$, $\Delta_i$ and $x_i\in \Delta_i^\times$ as above.
 We also fix a rank $1$ local system $\sL$ of  order $d$ defined by a character $ \chi_{\sL}: \pi_1^{\rm top}(X,x)\to \mu_d(\C)\subset \C^\times$.
 
 Let $T={\rm Spec}(B)$ be a  complex connected affine variety of finite type. A  $T$-{\it local system}  $\sV_T$  over $X$  is a local system of locally free $T$-modules, or equivalently a locally free $T$-sheaf $W$ together with a representation $\rho: \pi_1^{\rm top}(X,x)\to {\rm Aut}(W)$.  The {\it rank } of $\sV_T$ is the rank of $W$. A {\it geometrically  irreducible} $T$-{\it local system} is a $T$-local system  $\sV_T$ such that $\sV\times_T \bar \eta$ is an irreducible local system over $\bar \eta$ for all geometric generic points $\bar \eta$ of $T$. 
   We define a stack  from  the category of affine varieties  of finite type  over $K$ 
  to the category of groupoids, sending $T$ to  the groupoid of isomorphism classes of geometrically irreducible  $T$-local systems $\sV_T$ of rank $r$ together with an isomorphism $\wedge^r \sV_T\cong \sL\otimes_{\C} T$. 
  We denote it by ${\sf IrrLoc}(X,r, \sL)$. 
For each $i=1,\ldots, N$, we choose  a basis of $(\sV_T)_{x_{i}}$.
  We then define the  substack  $\underline{M}$ of ${\sf IrrLoc}(X,r, \sL)$ by the condition that the image of $T_{i}$ by the monodromy representation of $\sV_T|_{ \Delta^\times_{i}  }$ is a section of $\sK_i\times_K T\subset GL(r)\times_K T$.  As $\sK_i\subset GL(r)$ is locally closed, $\underline{M}\subset {\sf IrrLoc}(X,r, \sL)$ is a locally closed substack.

If  for any $i=1,\ldots, N$, the $\sK_i$ is the  conjugacy class of a quasi-unipotent matrix, we say that a point $[\sV]\in \underline{M}(K)$  has {\it quasi-unipotent monodromies along the components of $D$}.  This implies in particular that the varieties $\sK_i$  are defined over a number field $K$, and consequently $\underline{M}$ is defined over the same number field. 
 We shall need two properties of $\underline{M}$. 
\begin{proposition} \label{prop:stack}  Assume that  for any $i=1,\ldots, N$, the $\sK_i$ is the  conjugacy class of a quasi-unipotent matrix. Then
$\underline{M}$ is an algebraic stack of finite type defined over the number field $K$. In particular,  it  has finitely many $0$-dimensional irreducible components. 
\end{proposition}
\begin{proof}
As $\underline{M}$  is a locally closed substack of $ {\sf IrrLoc}(X,r, \sL),$
it suffices to show that the stack  ${\sf IrrLoc}(X,r,\sL)$ is an algebraic stack of finite type.

We recall the classical argument for this fact. First the automorphism group of any $\sV_T$ is $\mu_r(T)$,  which defines the finite constant group scheme $\mu_r$. Since the topological fundamental group of $X$ is a finitely presented group, we equivalently consider a finitely presented group $\Gamma$ and the stack ${\sf IrrRep}(\Gamma, \sL)$ of geometrically irreducible families of $\Gamma$-representations of rank $r$ given with an isomorphism of its determinant  to a fixed  $\sL$. At first we remark that this stack admits a fully faithful morphism to the stack of rank $r$ $\Gamma$-representations ${\sf Rep}(\Gamma, \sL)$ with an isomorphism of its $r$-th exterior power with $\sL:=\sO_T$ on which $\Gamma$ acts by $\chi_{\sL}$. 
 We claim that this morphism is an open immersion. To see this we have to prove that for a $T$-family of $\Gamma$-representations, that is, a rank $r$-representation $\rho\colon\Gamma \to {\rm Aut}(W)$, where $W$ is a locally free $B$-module of rank $r$, together with the determinant condition, there exists an open subset $T_0 \subset T$, such that the $\Gamma$-representation over a geometric point $x$ of $T$ is irreducible, if and only if $x$ is a geometric point of $T_0$. 

The $\Gamma$-representation $\rho$ induces an action of $\Gamma$ on the fibre bundle $\pi\colon \bigsqcup_{k = 0}^r\mathsf{Gr}(W,k) \to T$, where $\mathsf{Gr}(W,k)$ is the Grassmann bundle of $k$-planes in  $W$. We define $T_0$ as the complement of $\pi(\bigsqcup_{k = 0}^r\mathsf{Gr}(W,k)^{\Gamma})$. Since $\pi$ is proper, and the fixed point set is closed, we see that $T_0$ is an open subscheme of $T$. It is clear that for a geometric point $x$ of $T$, the induced representation $\rho_x$ is irreducible if and only if there does not exist an integer $k$, and a $k$-dimensional subspace fixed by $\Gamma$.

Since an open substack of a stack of finite type is also of finite type, we are now reduced to proving that ${\sf Rep}(\Gamma, \sL)$ is an algebraic stack of finite type. To see this we choose a presentation $\Gamma \simeq \langle r_1,\dots, r_e|s_1,\dots,s_f \rangle$ and $L_1, \ldots, L_e \in \mu_r(K)$, where $K$ is a number field,  such that $r_i\mapsto L_i$ defines the character $\chi_\sL$. 
Let $\mathcal{R}(\Gamma, \sL)$ be the affine $K$-variety of finite type, which represents the functor
$$T \mapsto  (A_1,\dots, A_e) \in GL(r)_K^e $$  such that the relations  
${\rm det} (A_j)=L_i$ for $ j=1,\ldots, e$ and 
$ s_i(A_1,\dots, A_e) = 1$ for $ i=1,\dots,f$ hold.
This is by definition an affine variety of finite type over $K$. 
To such a tuple $(A_1,\dots, A_e) $ one attaches the representation of $\Gamma$ on $\sO_T^{\oplus r}$ given by the $A_i$, and the isomorphism $e_1\wedge \ldots \wedge e_r \xrightarrow{} 1$ of the determinant of this representation with $\sL$. This construction induces an equivalence of the quotient stack $[\mathcal{R}(\Gamma, \sL)/SL(r)]$ with 
 ${\sf Rep}(\Gamma, \sL)$. 
Hence 
 ${\sf Rep}(\Gamma, \sL)$  is an algebraic stack of finite type over $K$.
\end{proof}

\begin{remark}
It follows from general theory that the stack $\underline{M}$ has a coarse moduli space. Indeed, according to \cite[Thm.~5.1.5]{ACV03}, there exists an algebraic stack $\underline{M}^{\mu_r}$, such that for any $z \in \underline{M}^{\mu_r}(T)$ we have $\mathsf{Aut}(z) = \{1\}$, and a morphism $\underline{M} \to \underline{M}^{\mu_r}$ which is the universal morphism to an algebraic stack with this property. This implies that for a $K$-scheme $T$,   the groupoid $\underline{M}^{\mu_r}(T)$ is equivalent to a set. Hence the algebraic stack $\underline{M}^{\mu_r}$ is actually (equivalent to) an algebraic space. The universal property of the morphism $\underline{M} \to \underline{M}^{\mu_r}$ shows that $\underline{M}^{\mu_r}$ is a coarse moduli space.
\end{remark}

The $\C$-points corresponding to $0$-dimensional components  are isolated points of the moduli space (so-called {\it rigid local systems}).  Thus they are all defined over a finite extension of $\mathbb{Q}$. 

\begin{proposition} \label{prop:coh}
The Zariski tangent space  $T_{[\sV]}$ at a point $[\sV]\in \underline{M}$  associated to $\sV$  defined over $K$ is  the finite dimensional $K$-vector space 
$H^1(U, a_* \sE nd^0(\sV))$.  In particular, if $H^1(U, a_* \sE nd^0(\sV))=0$, the geometrically irreducible $K$-local system is rigid, and there are finitely many such. 
\end{proposition}
\begin{proof} We follow Deligne's line of  proof.  Set $K[e]=K[e]/(e^2)$. By definition $T_{[\sV]} = \underline{M}( {\rm Spec}(K[e])$ where $[\sV]= \underline{M}( {\rm Spec}(K[e]\otimes_{K[e]} K)).$  So we want to identify $H^1(U, a_* \sE nd^0(\sV))$ with the set of isomorphism classes of $K[e]$-local systems $\sV_{K[e]}$ of rank $r$,
with an isomorphism $\wedge^r \sV_T\cong \wedge^r \sV\otimes_K T$ and 
with the following conditions: 
 \begin{itemize}
 \item[0)] $\sV_{K[e]} \otimes_{K[e]} K\cong \sV$;
 \item[1)]   for any complex point $y_i\in D_i\cap U$, and $\Delta_i$ a ball around $y_i$, 
 $\sV_{K[e]}|_{\Delta^\times_i} \cong \sV|_{\Delta^\times_i} \otimes_K K[e]$.
\end{itemize}
We have to show that the $K[e]$-local systems with a fixed isomorphism $\wedge^r \sV_T\cong \wedge^r \sV\otimes_K T$ and fulfilling the conditions 
0), 1) form a non-empty  torsor under $a_*\sE nd^0(\sV)$. 
So we choose a cover $X=\cup_\alpha \sX_\alpha$ by balls $\sX_\alpha$. On each $\sX_\alpha$, $\sV_{K[e]}|_{\sX_\alpha}$ and 
$\sV|_{\sX_\alpha}$ are trivialized, thus one has an isomorphism $\sV_{K[e]}|_{\sX_\alpha} \cong \sV|_{\sX_\alpha}\otimes_K K[e]$. The constraint 0)   implies  that on $\sV|_{\sX_\alpha \cap \sX_\beta}$ the composition of one such isomorphism by the inverse of the other one yields an isomorphism  in 
$(\sE nd^0(\sX_\alpha\cap \sX_\beta, \sV), +) =(1+  \sE nd^0(\sX_\alpha\cap \sX_\beta, \sV), \times)  \subset  {\rm Aut}_{K[e]} (\sX_\alpha\cap \sX_\beta, \sV\otimes_K K[e])$. This defines a cocycle, thus a cohomology class  in $
H^1(X, \sE nd^0(\sV))$. The constraint 1)  implies that on $\sX_\alpha\cap \sX_\beta \cap \Delta^\times_i$ the cohomology class defined by the cocycle is trivial. This shows that the cohomology class has values in 
\ga{}{ H^1(U, a_* \sE nd^0(\sV))= {\rm Ker}\big( H^1(X, \sE nd^0(\sV)) \to \oplus_{i=1}^N  \oplus_{j=1}^{j_i} H^1(\Delta_{ij}^\times, \sE nd^0(\sV))\big),\notag}
where for each $i=1,\ldots, N$, we covered $D_i$ by   $\cup_{j=1}^{j_i} \Delta_{ij}\supset D_i$ with $\Delta_{ij}\subset U$. 
 One trivially checks that the change of 
isomorphisms $\sV_{K[e]}|_{\sX_\alpha} \cong \sV|_{\sX_\alpha}\otimes_K K[e]$ changes the cocycle with values in $a_*\sE nd^0(\sV)$ by a coboundary. This finishes the proof of the cohomological part.  The rest follows directly from Proposition~\ref{prop:stack}.

\begin{remark} \label{rmk:IC}
Let $j_{!*} \sV$ be the {\it intermediate} extension on $\bar X$ of a $K$-local system $\sV$ on $X$.  From \cite[Prop.~2.1.11]{BBD82} one derives that  there is an exact triangle   $j_{!*} \sV\to Rb_* a_* \sV \to \sC$  in the bounded derived category of $\bar X$ such that $\sC$ is supported on $D_{\rm sing}$ and  is concentrated in degrees $\ge 2$. Thus it induces an isomorphism on $\H^1$. 
So Proposition~\ref{prop:coh} says $T_{[\sV]}=\H^1(\bar X, j_{!*} \sE nd^0(\sV))$.
 We also remark that if the  local monodromies of $\sV$ along the components of $D$ are finite, then 
 $j_{!*} \sV= j_* \sV, \ j_{!*}\sE nd^0 \sV= j_* \sE nd^0 \sV$.
\end{remark}

\end{proof}

\section{ Proof of Theorem~\ref{thm:int} } \label{sec:proof}

 Let $X $ be a smooth connected quasiprojective complex  variety, $x\in X$ be a geometric point, $j: X\hookrightarrow \bar X$ be a  good compactification,  thus $D=\bar X\setminus X$ is a strict normal crossings divisor. We write $j: X\xrightarrow{a} U=\bar X\setminus D_{\rm sing} \xrightarrow{b} \bar X$.  Let $\rho: \pi^{\rm top}_1(X,x)\to GL(r, \C)$ be a  complex linear representation, defining the local system $\sV$.

We fix a natural number $h$ and 
define the set $\sS(r,d, h)$ consisting of isomorphism classes of rank $r$   irreducible cohomologically rigid complex local systems  $\sV$ on $X$ with determinant of order dividing $d$,  such that the local monodromies at infinity, which are quasi-unipotent,  
have eigenvalues of order dividing $h$. 
 There are finitely many local systems of rank $1$ of order dividing $d$ and once the Jordan type of the 
 unipotent part of the 
monodromy representation along  $D_i$ is fixed, there are finitely many possibilities  for $\sK_i\subset GL(r, \C)$ (see notations of Section~\ref{sec:mod}).  As there are finitely many such Jordan types,  Proposition~\ref{prop:coh} implies that $\sS(r,d, h)$ is finite, 
of cardinality $\frak{N}=N(r,d,h)$. In addition, one has finitely many possibilities for $\sK_i$ and the determinant $\sL$, defining finitely many stacks  $\underline{M}_m$. 
 Let $K_0$ be a number field over which they are all defined are defined. 
 The disjoint union $\underline{N}=\sqcup_m \underline{M}_m$  is a stack of finite type defined over $K_0$.

There is a connected regular scheme $S$ of finite type over $\Z$  with a complex generic point ${\rm Spec} ( \C)\to S$ such that $( j: X \hookrightarrow \bar X, x,  D, D_J)$  is the base change from $S$ to ${\rm Spec}( \C)$  
of  $(j_S: X_S \hookrightarrow \bar X_S, 
x_S,  D_S=\bar X_S\setminus X_S,  D_{J,S} =\cap_{j\in J} D_{j,S})$  with the following properties.  The scheme $\bar X_S$ is smooth projective over $S$, $D_S$ is a relative normal crossings divisor with strata $D_{J,S}$,  $X_S=\bar X_S\setminus D_S$,    $x_S$ is a $S$-point of $X_S$.
This also defines $X_S\xrightarrow{a_S} U_S\xrightarrow{b_S} \bar X_S$.
We say for short that the objects with lower index $S$ are {\it models} over $S$ of the objects without lower index $S$ (which means that the latter ones are defined over $\C$).
\medskip

As the $\sV_i \in \sS(r,d, h)$ are cohomologically rigid,  they are in particular rigid.  As in addition the local monodromies at infinity are quasi-unipotent,  there is a number field $K \subset \C$  containing $K_0$ such that up to conjugacy  the underlying complex linear representations  factor as  $\rho_i: \pi^{\rm top}_1(X,x) \to GL(r, K) \to  GL(r, \C)$,   and such that the rank $1$  local systems ${\rm det}(\rho_i)$  factor through
$ \pi^{\rm top}_1(X,x) \to \mu_d(K)$.  As $\pi_1^{\rm top}(X,x)$ is finitely generated, there is a finite set $\Sigma$ of places of $K$ such that one has a factorization
\ga{}{\rho_i: \pi^{\rm top}_1(X,x) \xrightarrow{\rho_i^0} GL(r, \sO_{K,\Sigma} )  \to GL(r, K) \to  GL(r, \C), \notag} where $\sO_{K,\Sigma} \subset K$ is the ring of $\Sigma$-integers of $K$. 

\medskip

We fix a finite place $\lambda \in {\rm Spec}(\sO_{K,\Sigma})$, dividing  $ \ell \in {\rm Spec}(\Z)$.  We denote by  $K_\lambda$ the completion of $K$ at $\lambda$ and by 
$\sO_{K_\lambda}$  its ring of integers. This defines the representations 
\ga{}{\rho^{\rm top}_{i, \lambda}:  \pi^{\rm top}_1(X,x) \xrightarrow{\rho_i^{0, {\rm top}}} GL(r, \sO_{K,\Sigma} ) \to GL(r, \sO_{K_\lambda}) \notag}
with factorization
\ga{}{ \rho_{i, \lambda}:  \pi^{\rm \acute{e}t}_1(X,x)  \xrightarrow{\rho_{i,\lambda}^0}  GL(r, \sO_{K_\lambda}) \to GL(r, K_\lambda). \notag}
By extension of the ring of coefficients for  Betti cohomology, one has 
\ga{}{ 0=H^1(U,  a_*\sE nd^0(\rho_i))=H^1(U, a_*\sE nd^0(\rho^{0, {\rm top}}_i))\otimes_{\sO_{K,\Sigma}} \C, \notag}
thus $H^1(U, a_*\sE nd^0(\rho^{0, {\rm top}}_i))$ is torsion, 
while by comparison between Betti and \'etale cohomology one has 
\ga{}{0= H^1(U, a_*\sE nd^0(\rho^{0, {\rm top}}_i)) \otimes_{\sO_{K,\Sigma}} K_\lambda = H^1(U, a_*\sE nd^0(\rho_{i,\lambda})). \notag}
Let $\frak{m}_{K_\lambda} \subset \sO_{K_\lambda}$ be the maximal ideal, and 
\ga{}{ \overline{  \rho_{i, \lambda}^0}:  \pi^{\rm \acute{e}t}_1(X,x)  \to GL(r, \sO_{K_\lambda}/\frak{m}_{K_\lambda}) \notag}
be the residual representations. 

\medskip
We choose a closed point  $s\in S$ such that its  characteristic $p$  is prime to 
\begin{itemize} 
\item[]the cardinality of the residual monodromy groups 
$   \overline{ \rho_{i, \lambda}^0}(\pi^{\rm \acute{e}t}_1(X,x) ),$
\item[]
to  $d$,
\item[] to $\ell$, 
\item[] to the residual characteristics of the places in $\Sigma$,
\item[] and to $h$.

\end{itemize}

 There is a {\it specialization homomorphism}
\ga{}{sp: \pi_1^{\rm \acute{e}t}(X,x)\to \pi_1^{\rm \acute{e}t,p'}(X_{\bar s },x_{\bar s}) \notag}
defined by Grothendieck, 
   with target the  prime to $p$ quotient  of  $\pi_1^{\rm{ \acute{e}t}}(X_{\bar s },x_{\bar s})$, which is surjective,  and is an isomorphism on the prime to $p$ quotient of $\pi_1^{\rm \acute{e}t}(X,x)$  (\cite[X, Cor.~2.4]{Gro71}),  \cite[XIII,4.7]{Gro71}).  The complex point ${\rm Spec}(\C)\to S$  factors through the choice of a complex point ${\rm Spec}(\C)\to  T$ where $T={\rm Spec}( \bar \sO_{S,s})$  and  $\bar \sO_{S,s} $ is the strict henselization of $S$ at $s$.
   The  tame \'etale coverings of $X_{\bar s}$  lift to  $X_T$, defining a faithful functor from
   the category of  tame lisse sheaves on $X_{\bar s}$  to lisse sheaves on $X(\C)$.  This functor is an equivalence when restricted to  ``monodromy prime to $p$'' part of the fundamental group. That is,
  omitting the base points,  $sp$  comes from 
    the induced map 
$X\to X_T$ and  base change $
\pi_1^{\rm{ \acute{e}t,p'}}(X_{\bar s }) \xrightarrow{\cong}
 \pi_1^{\rm{\acute{e}t,p'}}(X_{  T  })$.

   The  
  push-out by $sp$ of the 
  homotopy exact sequence \cite[IX, Thm.~6.1]{Gro71}
\ga{hom}{1\to \pi_1^{\rm \acute{e}t}(X_{\bar s}, x_{\bar s})\to  \pi_1^{\rm \acute{e}t}(X_ s, x_s) \to 
\pi^{\rm \acute{e}t}_1(s,\bar s)\to 1\notag}
yields the ``prime to $p$ homotopy exact sequence''
\ga{hom}{1\to \pi_1^{\rm \acute{e}t,p'}(X_{\bar s}, x_{\bar s})\to  \pi_1^{\rm \acute{e}t,p'}(X_ s, x_s) \to 
\pi^{\rm \acute{e}t}_1(s,\bar s)\to 1\notag}
defining $\pi_1^{\rm \acute{e}t,p'}(X_ s, x_s)$.
As ${\rm Ker}(   \rho_{i, \lambda}^0( \pi^{\rm \acute{e}t}_1(X,x))  \to  
\overline{ \rho_{i, \lambda}^0}(\pi^{\rm \acute{e}t}_1(X,x) ))$ is a pro-$\ell$-group, by the choice of $s$,  one has a factorization 
\ga{sp}{
 \rho_{i, \lambda}^0:  \pi^{\rm \acute{e}t}_1(X,x)  \xrightarrow{ sp}  \pi^{\rm \acute{e}t, p'}_1(X_{\bar s},x_{\bar s}) 
\xrightarrow{\rho_{i, \lambda, \bar s}^0}
GL(r, \sO_{K_\lambda}), \notag
}
where 
and $\bar s$ is an $\bar \F_p$-point of $X_s$ above $x_s$. 

   This defines the  diagram
   \ga{}{\xymatrix{ 
 \ar@/_2.0pc/@[][rr]^{j_s}   X_s \ar[r]^{a_s} &  U_s \ar[r]^{b_s} & \bar X_s
     } \notag}
     and above it the  diagram
     \ga{}{\xymatrix{
 \ar@/_2.0pc/@[][rr]^{j_{\bar s}}   X_{\bar s} \ar[r]^{a_{\bar s}} &  U_{\bar s} \ar[r]^{b_{\bar s}} & \bar X_{\bar s}
     } \notag}
     Finally, still by the choice of $s$, one has a factorization
   \ga{}{ {\rm det}(\rho_i): \pi_1^{\rm top}(X,x)\to \pi^{\rm \acute{e}t}_1(X,x) \xrightarrow{sp} \pi^{\rm \acute{e}t,p'}_1(X_{\bar s},x_{\bar s}) \to \mu_d(K).\notag}
  We define  the representation $\rho_{i,\lambda, \bar s}: \pi_1^{\rm \acute{e}t,p'}(X_{\bar s}, x_{\bar s}) \xrightarrow{\rho^0_{i,\lambda, \bar s}} GL(r, \sO_{K_\lambda})  \to 
GL(r, K_\lambda).
$ 
The next proposition is a variant  Simpson's Theorem~\cite[Thm.~4]{Sim92}.
\begin{proposition} \label{prop:simpson} After replacing $s \in S$ by  any point $s' \in S(k')$, with $k(s)\subset k' \subset k(\bar s)$ with degree $s'/s$ sufficiently divisible, 
one has a factorization
\ga{}{\xymatrix{
 \ar[d]_{\rho_{i, \lambda, \bar s}}     \pi^{\rm \acute{e}t,p'}_1(X_{\bar s},x_{\bar s})  \ar[r] & \pi^{\rm \acute{e}t,p'}_1(X_{s},x_{ s}) 
  \ar[dl]^{\rho_{i, \lambda, s}} \\
GL(r, {K_\lambda})
} \notag}
 such that ${\rm det}(\rho_{i,\lambda, s})$ is finite. 
\end{proposition}
\begin{proof} 
The representation $\rho_{i,\lambda}$, or equivalently the representation $\rho_{i,\lambda, \bar s}$,
  defines a $K_\lambda$-point  $[\rho_{i,\lambda}]\in \underline{M}(K_\lambda)$. 
The point $x_s\in X_s$ is rational, thus splits the prime to $p$ homotopy exact sequence. We still denote by $g$ the lift to  $\pi_1^{\rm \acute{e}t}(X_s, x_s)$ of an element in $\pi_1^{\rm \acute{e}t}(s, \bar s)$.
For such a  $g$, we define the representation
\ga{}{ \rho^g_{i,\lambda}: \pi_1^{\rm \acute{e}t}(X,x)\to GL(r, K_\lambda) \notag \\
\gamma\mapsto \rho_{i,\lambda, \bar s}(g\cdot sp(\gamma) \cdot g^{-1}). \notag}
\begin{lemma} \label{lem:unip}
 $\rho_{i,\lambda}^g \in \underline{N}(K_\lambda)$. 
 \end{lemma}
 \begin{proof}
 We have to prove that the determinant of $\rho_{i,\lambda}^g$ 
 has order  dividing $d$, and that the monodromies along the components of $D$ at infinity 
 are quasi-unipotent with order of the eigenvalues dividing $h$.

We first show we may assume that $X$ is a curve. 
We take in $\bar X$ a smooth projective  curve  $\bar C$  which is a complete intersection of ample divisors, all of which containing $x$ and all the $y_\iota$,  in good position with respect to $D$.
 In particular, the curve $C$ contains $x$ and  the points $y_\iota$.  One may assume that $\bar C$ is defined over $S$, and that $y_{\iota,S}$ is a $S$-point of $D_{\iota,S}$.  Let $C=\bar C\cap X$. Via the  surjections $\pi^{\rm top}_1(C,x)\to  \pi_1^{\rm top}(X,x)$, $\pi^{\rm \acute{e}t}_1(C,x)\to  \pi_1^{\rm \acute{e}t}(X,x)$ inducing via the specialization the surjection
$\pi^{\rm \acute{e}t}_1(C_{\bar s},x_{\bar s})\to  \pi_1^{\rm \acute{e}t}(X_{\bar s},x_{\bar s})$, one reduces  the statement to the case where $X$ has dimension $1$.  

Let $k_{y_{\iota, s}}$  be the local field  which is  
 the field of fractions of the complete ring $\sO_{\bar X_s, y_{\iota,s}}$  of $\bar X_s$ at $y_{\iota, s}$. 
 A uniformizer $t$ of $\sO_{\bar X_s, y_{\iota,s}}$ yields an identification of the category of finite \'etale prime to $p$  extensions of 
 $k_{y_{\iota, s}}$  with the category of finite \'etale prime to $p$ covers of $\G_m$
  (see 
 \cite[Thm.~1.41]{Kat86} and \cite[Section~15]{Del89} for the theory of  tangential base points at infinity).  The rational point $1$ of $\G_m$ defines a fiber functor $1_t$, thus
  yields the prime to $p$ homotopy exact  exact sequence 
\ga{}{ 1\to \pi_1^{\rm \acute{e}t,p'}(k_{y_{\iota, \bar s}}, 1_t) \to \pi_1^{\rm \acute{e}t,p'}(k_{y_{\iota,  s}}, 1_t) \to 
\pi_1^{\rm \acute{e}t}(s, \bar s) \to 1 \notag}
for $k_{y_{\iota,  s}}$, together  
with  a splitting of $\pi_1^{\rm \acute{e}t}(s, \bar s)$ in  $\pi_1^{\rm \acute{e}t,p'}(k_{y_{\iota,  s}}, 1_t)$. 
It maps to the prime to $p$ homotopy exact  exact sequence 
\ga{}{ 1\to \pi_1^{\rm \acute{e}t,p'}(X_{\bar s}, 1_t) \to \pi_1^{\rm \acute{e}t,p'}(X_s, 1_t) \to 
\pi_1^{\rm \acute{e}t}(s, \bar s) \to 1 \notag}
for $X_s$ based at $1_t$, 
where $1_t$ is viewed as  a tangential base point on $X_s$.  An equivalence $\theta_s$ of fiber functors between $1_t$ and $x_{\bar s}$ for the category of  finite \'etale prime to $p$ covers of $X_s$ yields a  isomorphism of exact sequences 
from the  prime to $p$ homotopy exact  sequence  for $X_s$ based at $1_t$ with the one based at $x_{\bar s}$, compatibly with the section.  Thus conjugacy by $g$ stabilizes the image of 
$\pi_1^{\rm \acute{e}t,p'}(k_{y_{\iota, \bar s}}, 1_t) \to \pi_1^{\rm \acute{e}t,p'}(X_{\bar s}, 1_t) 
\xrightarrow{\theta_s} \pi_1^{\rm \acute{e}t,p'}(X_{\bar s}, x_{\bar s}).$  
Let $\sO_{\bar X, y_{\iota}}$  be the complete ring $\sO_{\bar X, y_{\iota}}$  of $\bar X$ at $y_{\iota}$, and  $k_{y_\iota}$ be its field of fractions. The specialization homomorphism for the fundamental groups based at $1_t$ sends 
$\pi_1^{\rm \acute{e}t}(k_{y_{\iota}}, 1_t)$  to $\pi_1^{\rm \acute{e}t,p'}(k_{y_{\iota, \bar s}}, 1_t)$, and the profinite completion of $\pi_1(\Delta^\times_\iota, 1_t)$ is  $\pi_1^{\rm \acute{e}t}(k_{y_{\iota}}, 1_t)$. Here we abused notations: $\G_m$ over $s$ lifts to $\G_m$ over $T$ then over $\C$, as well as the uniformizer $t$. We used the same notation $1_t$ for the base point on those lifts, and used that the category  of finite \'etale extensions  (resp. topological covers)
 of $k_{y_{\iota}}$  (resp. $\Delta^\times_\iota$) is equivalent to the corresponding one on $\G_m$.  Finally, identifying
 $\G_m(\C)$ with $\Delta_\iota^\times$ identifies $1_t$ and $x_\iota$ with two points in $\Delta_\iota^\times$. We choose a path $\theta_\iota$ between the two in
  $\Delta_\iota^\times$, and 
  a path  $\theta$  between $1_t$ and $x$ on $X$. This defines isomorphisms $\pi_1(\Delta^\times_\iota, 1_t)\xrightarrow{\theta_\iota} \pi_1(\Delta^\times_\iota, x_\iota)$ and $\pi_1(X, 1_t)\xrightarrow{\theta} \pi_1(X, x)$. Set
  $U= \theta_\iota^{-1}(T_\iota)$. 
  Then by assumption $\rho_i \circ \theta(U)\in \sK_\iota$. 
   Let $\hat U$ be its image  via the profinite completion  $\pi_1(\Delta_\iota^\times, 1_t)\to
  \pi_1^{\rm \acute{e}t}(k_{y_{\iota}}, 1_t) $.  Summarizing the information we have 
  \ga{}{\rho_{i,\lambda}^g\circ \theta(\hat U)=  \rho_{i, \lambda,\bar s}  (g\cdot sp( \theta(\hat U)) \cdot g^{-1})=
  \rho_{i,\lambda, \bar s} \circ \theta_s (g\cdot sp(\hat U) \cdot g^{-1})\notag}
  and
  \ga{}{ g\cdot sp(\hat U) \cdot g^{-1} \in \pi_1^{\rm \acute{e}t,p'}(k_{y_{\iota, \bar s}}, 1_t).
   \notag}
   As $sp(\hat U)$ is a topological generator of $\pi_1^{\rm \acute{e}t,p'}(k_{y_{\iota, \bar s}}, 1_t)\cong \hat \Z^{(p')}$,
  there is an element $(c_n) \in \hat \Z^{(p')}$, with $c_n \in \Z/n$, such that 
 $ g\cdot sp(\hat U) \cdot g^{-1} =sp (\hat U)^{c_n} \in \pi_1^{\rm \acute{e}t,p'}(k_{y_{\iota, \bar s}}, 1_t)\in \Z/n$.  As the subset  $A$ of matrices in $GL(r, K_\lambda)$ of quasi-unipotent matrices with order of the eigenvalues dividing $h$  and the order of the determinant dividing $d$ is closed, and  
$ \rho_{i,\lambda, \bar s} \circ \theta_s (sp(\hat U)^{c_n} )\in A$, one has $\rho_{i,\lambda}^g\circ \theta(\hat U) \in A$. 
Thus  the representation $\rho^{g}_{i,\lambda}$
defines a point $ [\rho^{g}_{i,\lambda}]\in \underline{N}(K_\lambda).$
\end{proof}
 The  map 
\ga{}{\pi_1^{\rm \acute{e}t}(s,\bar s)\to \underline{N}(K_\lambda), \ g\mapsto  [\rho^{g}_{i,\lambda}]\ \notag}
 is continuous for the profinite topology on $\pi_1^{\rm \acute{e}t}(s,\bar s)$
  and the $\lambda$-adic topology on $ \underline{N}(K_\lambda)$.
 As 
$[\rho_{i,\lambda}]\in \underline{N}(K_\lambda)$ is isolated, there is an open subgroup of $\pi^{\rm \acute{e}t}_1(s,\bar s)$  on which the map  is constant with image $[\rho_{i,\lambda}]$. This defines a point $s'$ with $\bar s\to s'\to s$ with  $\pi_1^{\rm \acute{e}t}(s', \bar s)$ being this open subgroup. We  abuse notations and set $s=s'$.  

\medskip

 Let $K_\lambda \subset \bar K_\lambda$ be an algebraic closure.
The representation  $\rho_{i,\lambda} \otimes \bar K_\lambda$  is  irreducible as $\rho_i$ is irreducible over $\C$. Thus the equation  $[\rho^g_{i, \lambda}]= [\rho_{i, \lambda}] \in \underline{N}(K_\lambda)$ 
implies that there is a $T(g)\in GL(r, K_\lambda)$ such that   $\rho_{i, \lambda, \bar s}(g\gamma g^{-1})=T(g) \rho_{i,\lambda, \bar s}(\gamma) T(g)^{-1}  \in GL(r, K_\lambda)$ for all $\gamma \in  
\pi_1^{\rm \acute{e}t}(X_{\bar s}, x_{\bar s})$, and  moreover $T(g)$ is uniquely defined up to multiplication by a scalar in $K_\lambda^\times$.  The so defined  map 
$\pi_1^{\rm \acute{e}t}(s,\bar s)\to   PGL(r, K_\lambda), \ g \mapsto \bar T(g)$, where $\bar T(g)$ is the image of $T(g)$,  is continuous for the profinite topology on $\pi_1^{\rm \acute{e}t}(s,\bar s)$ and the $\lambda$-adic topology on $PGL(r, K_\lambda)$.
Writing $\pi_1^{\rm \acute{e}t, p'}(X_s,x_s)$   as a semi-direct product of $\pi_1^{\rm \acute{e}t}(s,\bar s)$
by $\pi_1^{\rm \acute{e}t, p'}(X_{\bar s},x_{\bar s})$, 
we define $p\rho_{i,\lambda, s}: \pi_1(X_s, x_s)\to PGL(r, K_\lambda)$ by sending $(\gamma \cdot g)$ to 
$\rho_{i,\lambda,\bar s}(\gamma)\cdot \bar T (g)$.  It remains to lift $p\rho_{i,\lambda, s}$ to $\rho_{i,\lambda, s}$
as in the proposition. Our initial argument consisted in saying that the Brauer obstruction to the lift dies as $k(s)$ is finite, and  in then  using class field theory \cite[Prop.~1.3.4]{Del80}
 to ensure that the so constructed representation has finite determinant after twist by a character of $k(s)$.  We present here Deligne's argument which has the advantage to work on all base fields, but necessitates a new change of $s$. The representation $\rho_{i,\lambda, \bar s}$ has values in the subgroup $G\subset GL(r, K_\lambda)$ consisting of the elements with order $d'$ determinant, where $d'$ is any divisor of $d$. The composite  homomorphism $G\to GL(r, K_{\lambda}) \to PGL(r, K_\lambda)$ is finite \'etale onto its image.  Thus again base changing $s$ to $s'\to s$ finite with $\bar s\to s'$, the pull-back $\Pi$ of $\pi_1^{\rm \acute{e}t}(s', \bar s)$ in $\pi_1^{\rm \acute{e}t}(X_s, x_s)$ is open and the restriction of $p\rho_{i,\lambda,s}$ to  $\Pi$  lifts to $G$ 
 in a unique way so that on $\pi_1^{\rm \acute{e}t}(X_{\bar s}, x_{\bar s})$, it  is precisely $\rho_{i,\lambda, \bar s}$.

\end{proof}
We denote the lisse sheaves associated to $\rho_{i,\lambda}$  resp. $\rho_{i,\lambda,\bar s}$ resp. $\rho_{i,\lambda,  s}$
 by  $\sV_{i,\lambda}$  resp. $\sV_{i,\lambda, \bar s}$ resp. $\sV_{i,\lambda,  s}$.

 \begin{lemma} \label{lem:ext}
 The lisse sheaves $ \sV_{i,\lambda, \bar s}$, resp. $ \sV_{i,\lambda,  s}$
 are tame,
 have quasi-unipotent monodromies at infinity, and the order of the eigenvalues of the  local monodromies at infinity 
 divides $h$. 
 
  \end{lemma}

 \begin{proof}
 Since  $ \sV_{i,\lambda, \bar s}$ is defined by  a  representation of $\pi_1^{\rm \acute{e}t,p'}(X_{\bar s}, x_{\bar s})$  it is tame, so so is $ \sV_{i,\lambda,  s}$.
 We denote by  $\Gamma_{i,\lambda}\subset GL(r, \sO_{K_\lambda})$ the monodromy group of $\sV_{i,\lambda}$, which is  the one of $\sV_{i,\lambda,\bar s}$.   The rest of the statement is proven in the proof of 
Proposition~\ref{prop:simpson}.

 \end{proof}
 Fix any prime number $\ell'\neq p$ and denote by $n$ the dimension of $X$.
\begin{lemma} \label{lem:weights}  Let $\sA$ be a pure tame $\bar \Q_{\ell'}$-lisse sheaf of weight $0$ on $X_s$. 
Then
$$H^{1}(\bar X_{\bar s}, j_{ \bar s !*}\sA)=H^{1}(U_{\bar s}, a_{\bar s*}\sA)$$ and 
 is the $\bar \Q_{\ell '}$-sub vector space of $\oplus_{j\in \N} H^j(X_{\bar s}, \sA)$ consisting of all the elements of weight precisely $1$.
\end{lemma}
\begin{proof}

 The proof of Remark~\ref{rmk:IC} yields over $X_s$ the relation $H^{1}(\bar X_{\bar s}, j_{\bar s !*}\sA)=H^{1}(U_{\bar s}, a_{\bar s*}\sA)$.
 The weight of $H^{0}(X_{\bar s}, \sA)$ is $0$,  and for any $j$, 
 the weight of $H^j(X_{\bar s}, \sA)$ is $\ge j$ (see \cite[Thm.~3.3.1]{Del80}). Thus the weight $1$ part of 
 $\oplus_{j\in \N} H^j(X_{\bar s}, \sA)$ lies in
 $H^{1}(X_{\bar s}, \sA)$. 
One has a short   $\pi_1^{\rm \acute{e}t}(s, \bar s)$ equivariant exact sequence
\ga{}{ 0\to H^1(\bar X_{\bar s},  j_{\bar s !*} \sA)  \to H^1(X_{\bar s}, \sA)\to H^0(U_{\bar s}, R^1a_{\bar s *} \sA).\notag}
The group  $H^1(\bar X_{\bar s},  j_{\bar s !*} \sA) $ is pure of weight $1$ while $R^1a_{\bar s *} \sA$ has weights $\ge 2$ at closed points,  which is seen on curves, and on them  $\sA$ is tame (\cite[Thm.~1.1]{KS10}), thus  the local inertia $\Z_\ell(1)$ acts at the punctures at infinity. 
Thus a fortiori the weight of $H^0(U_{\bar s}, R^1a_{\bar s *} \sA)$ is $\ge 2$. 
 This finishes the proof.

\end{proof}

\begin{proof}[Proof of Theorem~\ref{thm:int}]
 We fix an embedding $K_{\lambda}\subset \bar \Q_{\ell}$.
 Let $\sigma: \bar \Q_\ell \to \bar \Q_{\ell'}$ be a field isomorphism, where $\ell'\neq p$. By Drinfeld's theorem \cite[Thm.~1.1]{Dri12}, there is a $\sigma$-companion $\sV_{i,\lambda, s}^\sigma$ to $ \sV_{i,\lambda,s}$. 
 By definition of the indexing, for  $i\in \{1,\ldots, \frak{N}\}$, the  complex local systems $\sV_i$   are irreducible and pairwise  non-isomorphic.   Thus  by definition, the $\bar \Q_\ell$ lisse sheaves $\sV_{i,\lambda,  s}$ 
   are irreducible and pairwise non-isomorphic.   If  $\sV_{i,\lambda, \bar s}^\sigma$ was not irreducible, it would split  after a finite base change  $s'\to s$ with $\bar s\to s'\to s$, thus on $X_{s'}$ , thus 
 the $\sigma^{-1}$-companion of $\sV^\sigma_{i,\lambda,  s'}$, which is $\sV_{i,\lambda, s'}$,  would split as well, a contradiction.  Likewise, if $\sV^{\sigma}_{i,\lambda, \bar s}$ is isomorphic to $\sV^{\sigma}_{j,\lambda, \bar s}$, since $H^0(X_{\bar s}, (\sV^{\sigma}_{i,\lambda, \bar s})^\vee \otimes \sV^{\sigma}_{j,\lambda, \bar s})$ has weight $0$, the isomorphism is defined over $X_s$, thus $i=j$.  Thus the $\bar \Q_\ell$ lisse sheaves $\sV_{i,\lambda,  \bar s}$ are irreducible and pairwise non-isomorphic.  
  In addition, since ${\rm det}(\sV_{i,\lambda, \bar s})^\sigma= {\rm det}(\sV^\sigma_{i,\lambda,\bar s}) $ is constructed by post-composing the $K_\lambda^\times$ character by $\sigma$, it has order precisely $d$.

 \medskip

 As $\sV_{i,\lambda, s}$ and $ \sV_{i,\lambda,s}^\sigma$ are pure of weight $0$, 
 so are $\sA_{is}=\sE nd^0( \sV_{i,\lambda, s})$ and $\sA_{is}^\sigma=\sE nd^0( \sV_{i,\lambda, s}^\sigma).$
 \medskip
 By local acyclicity  \cite[Lem.~3.14]{Sai17} applied to $X_T\to T$ used to define the specialization homomorphism, 
 \ga{}{ H^1(U_{\bar s},  a_{{\bar s}*} \sA_{i, {\bar s}} ) \to  H^1(U,  a_{*}\sA_{i }) \notag}
  is an isomorphism. Thus   $H^1(\bar X_{\bar s}, j_{!*}\sA_{i \bar s})=0.$    
  
  The $L$- functions  $L(X_{\bar s},  \sA_{i,s})$   and  
$L(X_{\bar s},  \sA_{i,s}^\sigma)$  defined by a product formula are equal  (\cite[5.2.3]{Del73}).  
In particular,  for any natural number $w$,  the pure weight $w$ summand of  $\oplus_j H^{j}_c(X_{\bar s}, \sA_{i,\bar s}) $ has the same dimension over $ \bar \Q_\ell$ as the pure weight $w$ summand of $\oplus_j H^{j}_c(X_{\bar s}, \sA^\sigma_{i, \bar s}) $ over $ \bar \Q_{\ell'}$.  By duality,  the same it true replacing 
the cohomologies $\oplus_j H^j_c(X_{\bar s}, \sA_{i, \bar s})$  and $\oplus_j H^j_c(X_{\bar s}, \sA^\sigma_{i, \bar s})$ by the cohomologies $\oplus_j H^j(X_{\bar s}, \sA_{i,\bar s})$  and $\oplus_j H^j(X_{\bar s}, \sA^\sigma_{i,\bar s})$.
Applying this for $w=1$, from Lemma~\ref{lem:weights}, we conclude 
$H^1(\bar X_{\bar s}, j_{\bar s!*}\sA^\sigma_{i \bar s})=  H^1(U_{\bar s}, a_{\bar s*}\sA^\sigma_{i, \bar s})=0.$

\medskip

Pulling back  along the specialization homomorphism
 $sp: \pi_1^{\rm \acute{e}t}( X, x)\to  \pi_1^{\rm \acute{e}t, p'}( X_{\bar s}, x_{\bar s})$ defines
 the $\bar \Q_{\ell '}$-lisse sheaves $\sV_{i \lambda}^\sigma$ and $\sA_i^\sigma$ on $X$,
together with the  specialization homomorphism
 \ga{}{ H^1(U_{\bar s},  a_{{\bar s}*} \sA^\sigma_{i, {\bar s}} ) \to  H^1(U,  a_{*}\sA^\sigma_{i }) .\notag}
By local acyclicity again,  it is an isomorphism thus $H^1(U,  a_{*}\sA^\sigma_{i })=0$.

\medskip

  We now define $\sV_i^{ \sigma \ \rm top}$ to be the $\bar \Q_{\ell'}$-local system on $X$ which is defined by composing the representation with the homomorphism  $\pi_1^{\rm top}(X,x) \to \pi^{\rm \acute{e}t}_1(X,x)$. By the comparison between Betti and \'etale cohomology one has  $H^1(U, a_* \sE nd(\sV_i^{ \sigma  \ \rm top}))=0$. 
  Furthermore,   the $\sV_i^{ \sigma \ \rm top}$ on $X$  are irreducible and  pairwise non-isomorphic, and  ${\rm det} (\sV_i^{ \sigma \ \rm top})$ has order $d$. 

\medskip
On the other hand, as $\sV_{i,\lambda, s}$ is tame by Lemma~\ref{lem:ext}, 
it is tame in restriction to all curves (\cite[Thm.~1.1]{KS10}). 
Taking a smooth curve $\bar C\subset \bar X$ which is a complete intersection of smooth ample divisors in good position with respect to $\bar X\setminus X,$ and denoting by $C$ its intersection with $X$, 
 the restriction  $\sV_{i, \lambda, \bar s}|_{C_{\bar s}}$ has the same monodromy group as the one of $\sV_{i, \lambda, \bar s}$, thus has quasi-unipotent monodromies at the points $C_{\bar s}\cap D_{\bar s}$, and their eigenvalues have order bounded by $h$ . By  \cite[Thm.~9.8]{Del73b},  $\sV^\sigma_{i, \lambda, \bar s}$  has the same property.  Thus $ \sV_i^{ \sigma \ \rm top}$ has quasi-unipotent monodromies along the components of $\bar X\setminus X$, with order of the eigenvalues bounded by $h$.  (There is a slight abuse of notations here,  the curve $\bar C$ chosen is perhaps not defined over the same $S$, we just take a $s$ in the construction on which $C$ is defined).

\medskip

As $\pi_1^{\rm top}(X,x)$ is a finitely generated group, there is a subring $A\subset \bar \Q_{\ell' }$ of finite type such that the monodromy representations of the $ \sV^{\sigma_i \ {\rm top}}$ factor through
$\pi_1^{\rm top} (X,x)\to GL(r, A)$. We fix a complex embedding $A\hookrightarrow \C$.  This defines the complex local systems $\sV_i^{\sigma \ \C}$. By definition, there are irreducible, pairwise different, have determinant of order precisely $d$ and the eigenvalues of the monodromies at infinity have order at most $h$. 
 By comparison between Betti and \'etale cohomology, they are cohomologically rigid. 
Thus the set of isomorphism classes of the 
$\sV_i^{\sigma \ \C}$ is precisely $\sS(r,d,  h)$. On the other hand, they are integral at all places of number rings in $\bar \Z_{\ell '}$ which divide $\ell'$. We conclude the proof by doing the construction for all $\ell'$ divided by places in $\Sigma$.

\end{proof}

\begin{remark} \label{rmk:deligne}
This remark is due to Pierre Deligne.  If in Theorem~\ref{thm:int}, the irreducible complex local system $\sV$ with quasi-unipotent monodromies along the components of $D$   is assumed to be
orthogonal, then it is integral under the assumption $\H^1(\bar X, j_{!*} \wedge^2 \sV)=0$. This assumption is weaker than the assumption $\H^1(\bar X, j_{!*} \sE nd^0(\sV))=0$ of Theorem~\ref{thm:int}, as, as $\sV$ is self-dual, 
$\wedge^2 \sV$ is a summand of $\sE nd^0(\sV)$, and thus 
$\H^1(\bar X, j_{!*} \wedge^2 \sV)$ is a summand of $\H^1(\bar X, j_{!*} \sE nd^0(\sV))$.
 Likewise, if 
$\sV$ is assumed to be symplectic, then it is integral under the assumption $\H^1(\bar X, j_{!*}  Sym^2 \sV)=0$. 
Again   $ Sym^2 \sV$ is a summand   $\sE nd^0(\sV)$. 
We do not detail the proof. One has to  perform the stack construction in Section~\ref{sec:mod} for the corresponding category of orthogonal, resp. symplectic local systems, and eventually see that the companion construction in Section~\ref{sec:proof} preserves this category. 
\end{remark}

\end{document}